\documentclass{article}
\usepackage{amsmath,amsthm,amsopn,amstext,amscd,amsfonts,amssymb,verbatim}
\usepackage{pdflscape}
\usepackage{natbib}

\def\diag{\mathop{\rm diag}\nolimits}
\def\tr{\mathop{\rm tr}\nolimits}

\def\build#1#2#3{\mathrel{\mathop{#1}\limits^{#2}_{#3}}}

\def \Vol {\mathop{\rm Vol}\nolimits}
\def \re {\mathop{\rm Re}\nolimits}

\def\etr{\mathop{\rm etr}\nolimits}

\renewenvironment{abstract}
                 {\vspace{6pt}
                  \begin{center}
                  \begin{minipage}{5in}
                  \centerline{\textbf{Abstract}}
                  \noindent\ignorespaces
                 }
                 {\end{minipage}\end{center}}
\newtheorem{thm}{\textbf{Theorem}}[section]
\newtheorem{cor}{\textbf{Corollary}}[section]

\newtheorem{prop}{\textbf{Proposition}}[section]
\theoremstyle{definition}

\title{\Large \textbf{Jacobians of singular matrix transformations: Extensions}}
\author{
  \textbf{Jos\'e A. D\'{\i}az-Garc\'{\i}a} \thanks{Corresponding author\newline
    }\\
  {\normalsize Department of Statistics and Computation} \\
  {\normalsize 25350 Buenavista, Saltillo, Coahuila, Mexico} \\
  {\normalsize E-mail: jadiaz@uaaan.mx} \\[2ex]
  \textbf{Ram\'on Guti\'errez-S\'anchez} \\
  {\normalsize Department of Statistics and O.R} \\
  {\normalsize University of Granada} \\
  {\normalsize Granada 18071, Spain}\\
  {\normalsize E-mail: ramongs@ugr.es}\\
}
\date{}
\begin{document}
\maketitle

\begin{abstract}
This article presents a unified approach to simultaneously compute the Jacobians of several
singular matrix transformations in the real, complex, quaternion and octonion cases. Formally,
these Jacobians are obtained for real normed division algebras with respect to the Hausdorff
measure.
\end{abstract}

\section{Introduction}\label{sec1}

A fundamental tool in statistical theory and in particular in distribution theory is the
computation of Jacobians of matrix transformations. Many such Jacobians were first found in the
real case and subsequently for complex cases; among many others, see \citet{w:56};
\citet[Sections 4 and 8]{j:64}; \citet{k:65}, \citet{m:82}; \citet{rva:05a} and \citet{m:97}
for a detailed review of this topic. Most of these Jacobians were obtained with respect to the
Lebesgue measure. However, several recent articles have examined these Jacobian for the
singular random matrix case. i.e., they study densities and then calculate the Jacobians of the
matrix transformations with respect to the Hausdorff measure, see \citet{k:68}, \citet{u:94},
\citet{dgm:97}, \citet{dg:97}, \citet{s:03}, \citet{dgg:04}, \citet{dggg:05a},
\-\citet{dggg:05b}, \citet{ihl:07} and \citet{dg:07}, among others. In addition, \citet{rv:05}
studied some of these latter Jacobians in the singular complex case and \citet{lx:10}
considered the singular quaternion case.

Using results obtained from abstract algebra, it is possible to demonstrate a unified means of
addressing the computation of Jacobians in the singular and nonsingular real, complex,
quaternion and octonion cases, simultaneously. This approach has been used for some time in
random matrix theory in the nonsingular case, see \citet{er:05} and \citet{f:05}.

For the sake of completeness, in the present work we consider the case of octonions, but it
should be noted that many results for the octonion case can only be conjectured, because there
remain many unresolved theoretical problems in this respect, see \citet{dm:99}. Furthermore, as
stated by \citet{b:02}, the relevance of the octonion case for understanding the real world has
yet to be clarified.

The rest of this paper is structured as follows: Section \ref{sec2} provides some definitions
and notation on real normed division algebras, showing some ideas about the explicit form of
the Hasusdorff measure. The main results on the computation of Jacobians for singular matrix
transformations are presented in Section \ref{sec3}. It is emphasised that all these results
are obtained for real normed division algebras.

\section{Notation and real normed division algebras}\label{sec2}

Let us introduce some notation and useful results. A detailed discussion of real normed
division algebras may be found in \citet{b:02}. For convenience, we shall introduce some
notation, although in general we adhere to standard forms.

For the purposes of this study, a \textbf{vector space} is always a finite-dimensional module
over the field of real numbers. An \textbf{algebra} $\mathfrak{F}$ is a vector space that is
equipped with a bilinear map $m: \mathfrak{F} \times \mathfrak{F} \rightarrow \mathfrak{F}$
termed \emph{multiplication} and a nonzero element $1 \in \mathfrak{F}$ termed the \emph{unit}
such that $m(1,a) = m(a,1) = 1$. As usual, we abbreviate $m(a,b) = ab$ as $ab$. We do not
assume $\mathfrak{F}$ associative. Given an algebra, we freely think of real numbers as
elements of this algebra via the map $\omega \mapsto \omega 1$.

An algebra $\mathfrak{F}$ is a \textbf{division algebra} if given $a, b \in \mathfrak{F}$ with
$ab=0$, then either $a=0$ or $b=0$. Equivalently, $\mathfrak{F}$ is a division algebra if the
operation of left and right multiplications by any nonzero element is invertible. A
\textbf{normed division algebra} is an algebra $\mathfrak{F}$ that is also a normed vector
space with $||ab|| = ||a||||b||$. This implies that $\mathfrak{F}$ is a division algebra and
that $||1|| = 1$.

There are exactly four normed division algebras: real numbers ($\Re$), complex numbers
($\mathfrak{C}$), quaternions ($\mathfrak{H}$) and octonions ($\mathfrak{O}$), see
\citet{b:02}. We take into account that $\Re$, $\mathfrak{C}$, $\mathfrak{H}$ and
$\mathfrak{O}$ are the only normed division algebras; moreover, they are the only alternative
division algebras, and all division algebras have a real dimension of $1, 2, 4$ or $8$, which
is denoted by $\beta$, see \citet[Theorems 1, 2 and 3]{b:02}. In other branches of mathematics,
parameters $\alpha = 2/\beta$  and $t = \beta/4$ are used, see Table \ref{table1} and
\citet{er:05} and \citet{k:84}.

\medskip

\begin{table}[!h]
  \caption{Values of $\beta = 2/\alpha$ and $t = \beta/4$ parameters.}\label{table1}
  \renewcommand\arraystretch{1}
\begin{center}
\begin{tabular}{l|l l l}
  \hline\hline
  $\beta$ & $\alpha$ & $t$ & $\quad$Normed divison algebra \\
  \hline
  1 &   2 & 1/4 & real ($\Re$) \\
  2 &   1 & 1/2 & complex ($\mathfrak{C}$) \\
  4 & 1/2 &   1 & quaternionic ($\mathfrak{H}$) \\
  8 & 1/4 &   2 & octonion ($\mathfrak{O}$) \\
  \hline
\end{tabular}
\end{center}
\end{table}

Let $\mathfrak{L}^{\beta}_{m,n}$ be the linear space of all $n \times m$ matrices of rank $m
\leq n$ over $\mathfrak{F}$ with $m$ distinct positive singular values, where $\mathfrak{F}$
denotes a \emph{real finite-dimensional normed division algebra}. Let $\mathfrak{F}^{n \times
m}$ be the set of all $n \times m$ matrices over $\mathfrak{F}$. The dimension of
$\mathfrak{F}^{n \times m}$ over $\Re$ is $\beta mn$. Let $\mathbf{A} \in \mathfrak{F}^{n
\times m}$, then $\mathbf{A}^{*} = \overline{\mathbf{A}}^{T}$ denotes the usual conjugate
transpose.

The set of matrices $\mathbf{H}_{1} \in \mathfrak{F}^{n \times m}$ such that
$\mathbf{H}_{1}^{*}\mathbf{H}_{1} = \mathbf{I}_{m}$ is a manifold denoted ${\mathcal
V}_{m,n}^{\beta}$, and is termed the \emph{Stiefel manifold} ($\mathbf{H}_{1}$ is also known as
\emph{semi-orthogonal} ($\beta = 1$), \emph{semi-unitary} ($\beta = 2$), \emph{semi-symplectic}
($\beta = 4$) and \emph{semi-exceptional type} ($\beta = 8$) matrices, see \citet{dm:99}). The
dimension of $\mathcal{V}_{m,n}^{\beta}$ over $\Re$ is $[\beta mn - m(m-1)\beta/2 -m]$. In
particular, ${\mathcal V}_{m,m}^{\beta}$ with dimension over $\Re$, $[m(m+1)\beta/2 - m]$, is
the maximal compact subgroup $\mathfrak{U}^{\beta}(m)$ of $\mathfrak{L}^{\beta}_{m,m}$ and
consists of all matrices $\mathbf{H} \in \mathfrak{F}^{m \times m}$ such that
$\mathbf{H}^{*}\mathbf{H} = \mathbf{I}_{m}$. Therefore, $\mathfrak{U}^{\beta}(m)$ is the
\emph{real orthogonal group} $\mathcal{O}(m)$ ($\beta = 1$), the \emph{unitary group}
$\mathcal{U}(m)$ ($\beta = 2$), \emph{compact symplectic group} $\mathcal{S}p(m)$ ($\beta = 4$)
or \emph{exceptional type matrices} $\mathcal{O}o(m)$ ($\beta = 8$), for $\mathfrak{F} = \Re$,
$\mathfrak{C}$, $\mathfrak{H}$ or $\mathfrak{O}$, respectively.

We denote by ${\mathfrak S}_{m}^{\beta}$ the real vector space of all $\mathbf{S} \in
\mathfrak{F}^{m \times m}$ such that $\mathbf{S} = \mathbf{S}^{*}$. Let
$\mathfrak{P}_{m}^{\beta}$ be the \emph{cone of positive definite matrices} $\mathbf{S} \in
\mathfrak{F}^{m \times m}$; then $\mathfrak{P}_{m}^{\beta}$ is an open subset of ${\mathfrak
S}_{m}^{\beta}$. Over $\Re$, ${\mathfrak S}_{m}^{\beta}$ consist of \emph{symmetric} matrices;
over $\mathfrak{C}$, \emph{Hermitian} matrices; over $\mathfrak{H}$, \emph{quaternionic
Hermitian} matrices (also termed \emph{self-dual matrices}) and over $\mathfrak{O}$,
\emph{octonionic Hermitian} matrices. Generically, the elements of $\mathfrak{S}_{m}^{\beta}$
are termed \textbf{Hermitian matrices}, irrespective of the nature of $\mathfrak{F}$. The
dimension of $\mathfrak{S}_{m}^{\beta}$ over $\Re$ is $[m(m-1)\beta+2]/2$.

\noindent Let $\mathfrak{D}_{m}^{\beta}$ be the \emph{diagonal subgroup} of
$\mathcal{L}_{m,m}^{\beta}$ consisting of all $\mathbf{D} \in \mathfrak{F}^{m \times m}$,
$\mathbf{D} = \diag(d_{1}, \dots,d_{m})$ and let $\mathfrak{T}${\tiny U}$_{m,n}^{+\beta}$ be
the \emph{semi-upper-triangular subgroup} of $\mathcal{L}_{m,n}^{\beta}$ consisting of all
$\mathbf{T} \in \mathfrak{F}^{n \times m}$, with $t_{ii} > 0$.

Now, let $\mathfrak{L}_{m,n}^{+ \beta}(q)$ be the linear space of all $n \times m$ matrices of
rank $q \leq \min(n,m)$ with $q$ distinct singular values and let $\mathfrak{S}_{m}^{+
\beta}(q)$, the $(\beta mq - \beta q(q - 1)/2)- q$-dimensional manifold of rank $q$ of positive
semidefinite matrices $\mathbf{S} \in \mathfrak{F}^{m \times m}$ with $q$ distinct positive
eigenvalues.

For any matrix $\mathbf{X} \in \mathfrak{F}^{n \times m}$, $d\mathbf{X}$ denotes the\emph{
matrix of differentials} $(dx_{ij})$, and we denote the \emph{measure} or volume element as
$(d\mathbf{X})$ when $\mathbf{X} \in \mathfrak{F}^{m \times n}$, see \citet{d:02}.

If $\mathbf{X} \in \mathfrak{F}^{n \times m}$ then $(d\mathbf{X})$ (the Lebesgue measure in
$\mathfrak{F}^{n \times m}$) denotes the exterior product of the $\beta mn$ functionally
independent variables
$$
  (d\mathbf{X}) = \bigwedge_{i = 1}^{n}\bigwedge_{j = 1}^{m}dx_{ij} \quad \mbox{ where }
    \quad dx_{ij} = \bigwedge_{k = 1}^{\beta}dx_{ij}^{(k)}.
$$

If $\mathbf{S} \in \mathfrak{S}_{m}^{\beta}$ (or $\mathbf{S} \in \mathfrak{T}${\tiny
U}$_{m,n}^{+\beta}$) then $(d\mathbf{S})$ (the Lebesgue measure in $\mathfrak{S}_{m}^{\beta}$
or in $\mathfrak{T}${\tiny U}$_{m,n}^{+\beta}$) denotes the exterior product of the
$m(m+1)\beta/2$ functionally independent variables (or denotes the exterior product of the
$m(m-1)\beta/2 + n$ functionally independent variables, if $s_{ii} \in \Re$ for all $i = 1,
\dots, m$)
$$
  (d\mathbf{S}) = \left\{
                    \begin{array}{ll}
                      \displaystyle\bigwedge_{i \leq j}^{m}\bigwedge_{k = 1}^{\beta}ds_{ij}^{(k)}, &  \\
                      \displaystyle\bigwedge_{i=1}^{m} ds_{ii}\bigwedge_{i < j}^{m}\bigwedge_{k = 1}^{\beta}ds_{ij}^{(k)}, &
                       \hbox{if } s_{ii} \in \Re.
                    \end{array}
                  \right.
$$
The context generally establishes the conditions on the elements of $\mathbf{S}$, that is, if
$s_{ij} \in \Re$, $\in \mathfrak{C}$, $\in \mathfrak{H}$ or $ \in \mathfrak{O}$. It is
considered that
$$
  (d\mathbf{S}) = \bigwedge_{i \leq j}^{m}\bigwedge_{k = 1}^{\beta}ds_{ij}^{(k)}
   \equiv \bigwedge_{i=1}^{m} ds_{ii}\bigwedge_{i < j}^{m}\bigwedge_{k =
1}^{\beta}ds_{ij}^{(k)}.
$$
Observe, too, that for the Lebesgue measure $(d\mathbf{S})$ defined thus, it is required that
$\mathbf{S} \in \mathfrak{P}_{m}^{\beta}$, that is, $\mathbf{S}$ must be a non singular
Hermitian matrix (Hermitian positive definite matrix).

If $\mathbf{\Lambda} \in \mathfrak{D}_{m}^{\beta}$ then $(d\mathbf{\Lambda})$ (the Lebesgue
measure in $\mathfrak{D}_{m}^{\beta}$) denotes the exterior product of the $\beta m$
functionally independent variables
$$
  (d\mathbf{\Lambda}) = \bigwedge_{i = 1}^{n}\bigwedge_{k = 1}^{\beta}d\lambda_{i}^{(k)}.
$$

In addition, observe that, if $\mathbf{X} \in \mathfrak{L}_{m,n}^{+ \beta}(q)$, we can write
$\mathbf{X}$ as
$$
  \mathbf{X}_{1} = \left (
          \begin{array}{cc}
            \build{\mathbf{X}_{11}}{}{q \times q} &  \build{\mathbf{X}_{12}}{}{q \times m-q} \\
            \build{\mathbf{X}_{21}}{}{n-q \times q} &  \build{\mathbf{X}_{11}}{}{n-q \times m-q} \\
          \end{array}
          \right )
$$
such that $r(\mathbf{X}_{11}) = q$. This is equivalent to the right product of the matrix
$\mathbf{X}$ with a permutation matrix $\mathbf{\Pi}$, see \citet[section 3.4.1, 1996]{gvl:96},
that is $\mathbf{X}_{1} = \mathbf{X}\mathbf{\Pi}$. Note that the exterior product of the
elements from the differential matrix $d\mathbf{X}$ is not affected by the fact that we
multiply $\mathbf{X}$ (right or left) by a permutation matrix, that is, $(d\mathbf{X}_{1}) =
(d(\mathbf{X}\mathbf{\Pi})) = (d\mathbf{X})$, since $\mathbf{\Pi} \in \mathfrak{U}^{\beta}(m)$,
see \citet[Section 2.1, 1982]{m:82} and \citet{j:54}. Then, without loss of generality,
$(d\mathbf{X})$ will be defined as the exterior product of the differentials $dx_{ij}$, such
that $x_{ij}$ are mathematically independent. It is important to note that we will have $(nq+mq
-q^{2})\beta$ mathematically independent elements in the matrix $\mathbf{X} \in {\mathcal
L}_{m,n}^{+\beta}(q)$, corresponding to the elements of  $\mathbf{X}_{11}, \mathbf{X}_{12}$ and
$\mathbf{X}_{21}$. Explicitly,
\begin{equation}\label{X}
    (d\mathbf{X}) \equiv (d\mathbf{X}_{11})\wedge(d\mathbf{X}_{12})\wedge(d\mathbf{X}_{21}) =
            \bigwedge_{i=1}^{n}\bigwedge_{j=1}^{q}\bigwedge_{k = 1}^{\beta}dx_{ij}^{(k)}
            \bigwedge_{i=1}^{q} \bigwedge_{j = q+1}^{m}\bigwedge_{k = 1}^{\beta}dx_{ij}^{(k)}
\end{equation}

Similarly, given $\mathbf{S} \in \mathfrak{S}_{m}^{+ \beta}(q)$, we define $(d\mathbf{S})$ as
$$
  (d\mathbf{S}) = \left\{
                    \begin{array}{ll}
                      \displaystyle\bigwedge_{i=1}^{q}\bigwedge_{j = i}^{m}
                      \bigwedge_{k = 1}^{\beta}ds_{ij}^{(k)}, &  \\
                      \displaystyle\bigwedge_{i=1}^{q} ds_{ii}\bigwedge_{i=1}^{q}\bigwedge_{j = i+1}^{m}
                      \bigwedge_{k = 1}^{\beta}ds_{ij}^{(k)}, &
                       \hbox{if } s_{ii} \in \Re.
                    \end{array}
                  \right.
$$
The context generally establishes the conditions on the elements of $\mathbf{S}$, that is, if
$s_{ij} \in \Re$, $\in \mathfrak{C}$, $\in \mathfrak{H}$ or $ \in \mathfrak{O}$. It is
considered that
$$
  (d\mathbf{S}) = \bigwedge_{i=1}^{q}\bigwedge_{j = i}^{m}\bigwedge_{k = 1}^{\beta}ds_{ij}^{(k)}
   \equiv \bigwedge_{i=1}^{q} ds_{ii}\bigwedge_{i=1}^{q}\bigwedge_{j = i+1}^{m}\bigwedge_{k = 1}^{\beta}ds_{ij}^{(k)}.
$$
Again, we should note that, for this case, the matrix $\mathbf{S}$ can be written as
$$
   \mathbf{S} \equiv \left (
          \begin{array}{cc}
            \build{\mathbf{S}_{11}}{}{q \times q} & \build{\mathbf{S}_{12}}{}{q \times m-q} \\
            \build{\mathbf{S}_{21}}{}{m-q \times q} & \build{\mathbf{S}_{22}}{}{m-q \times m-q} \\
          \end{array}
       \right ) \qquad \mbox{with}  \qquad r(\mathbf{S}_{11}) = q.
$$
such that the number of mathematically independent elements in $\mathbf{S}$ is $\beta mq -q(q +
1)\beta/2 - q$ corresponding to the mathematically independent elements of $\mathbf{S}_{12}$
and $\mathbf{S}_{11}$. Recall that $\mathbf{S}_{11} \in \mathfrak{P}_{q}^{\beta}$, in such a
way that $\mathbf{S}_{11}$ has $q(q -1)\beta/2 + q$, therefore,
$$
    (d\mathbf{S}) \equiv (d\mathbf{S}_{11})\wedge(d\mathbf{S}_{12}).
$$

Observe that an explicit form for $(d\mathbf{X})$ and $(d\mathbf{S})$ depends on the
factorisation (base and coordinate set) employed to represent $\mathbf{X}$ or $\mathbf{S}$,
that is, they depend on the measure factorisation of $(d\mathbf{X})$ and $(d\mathbf{S})$. For
example, by using the nonsingular part of the decomposition in singular values and the
nonsingular part of the spectral decomposition for $\mathbf{X}$ and $\mathbf{S}$, then we can
find an explicit form for $(d\mathbf{X})$ and $(d\mathbf{S})$ (see Propositions \ref{lemsvd}
and \ref{lemsd}, respectively), which are not unique, see \citet{k:68}, \citet{dgm:97},
\citet{u:94} and \citet{dg:97}. Alternatively, an explicit form for $(d\mathbf{X})$ and
$(d\mathbf{S})$ can be found in terms of the QR and Cholesky decompositions, see Propositions
\ref{teoqr} and \ref{teoch}, respectively.

A singular random matrix $\mathbf{X}$ in  $\mathfrak{L}_{m,n}^{+ \beta}(q)$ or
$\mathfrak{S}_{m}^{+\beta}(q)$ does not have a density with respect to Lebesgue's measure in
$\mathfrak{F}^{n \times m}$, but it does possess a density on a subspace ${\mathcal M} \subset
\mathfrak{F}^{n \times m}$; see \citet{k:68}, \citet[p. 527]{r:73}, \citet{dgm:97},
\citet{u:94} and \citet[p. 297]{c:99}. Formally, $\mathbf{X}$ has a density with respect to
Hausdorff's measure, which coincides with Lebesgue's measure, when the latter is defined on the
subspace ${\mathcal M}$; see \citet[p. 247]{b:86}, \citet{u:94}, \citet{dgm:97}, \citet{dgg:05}
and \citet{dggj:06}.

The surface area or volume of the Stiefel manifold $\mathcal{V}^{\beta}_{m,n}$ is
\begin{equation}\label{vol}
    \Vol(\mathcal{V}^{\beta}_{m,n}) = \int_{\mathbf{H}_{1} \in
  \mathcal{V}^{\beta}_{m,n}} (\mathbf{H}^{*}_{1}d\mathbf{H}_{1}) =
  \frac{2^{m}\pi^{mn\beta/2}}{\Gamma^{\beta}_{m}[n\beta/2]},
\end{equation}
where $\mathbf{H} = (\mathbf{H}_{1}|\mathbf{H}_{2}) = (\mathbf{h}_{1}, \dots,
\mathbf{h}_{m}|\mathbf{h}_{m+1}, \dots, \mathbf{h}_{n}) \in \mathfrak{U}^{\beta}(m)$. It can be
proved that this differential form does not depend on the choice of the $\mathbf{H}_{2}$
matrix. When $m = 1$; $\mathcal{V}^{\beta}_{1,n}$ defines the unit sphere in
$\mathfrak{F}^{n}$. This is, of course, an $(n-1)\beta$- dimensional surface in
$\mathfrak{F}^{n}$. When $m = n$ and denoting $\mathbf{H}_{1}$ by $\mathbf{H}$,
$(\mathbf{H}^{*}d\mathbf{H})$ is termed the \emph{Haar measure} on $\mathfrak{U}^{\beta}(m)$.
Also, $\Gamma^{\beta}_{m}[a]$ denotes the multivariate Gamma function for the space
$\mathfrak{S}_{m}^{\beta}$, and is defined by
\begin{eqnarray*}
  \Gamma_{m}^{\beta}[a] &=& \displaystyle\int_{\mathbf{A} \in \mathfrak{P}_{m}^{\beta}}
  \etr\{-\mathbf{A}\} |\mathbf{A}|^{a-(m-1)\beta/2 - 1}(d\mathbf{A}) \\
&=& \pi^{m(m-1)\beta/4}\displaystyle\prod_{i=1}^{m} \Gamma[a-(i-1)\beta/2],
\end{eqnarray*}
where $\etr(\cdot) = \exp(\tr(\cdot))$, $|\cdot|$ denotes the determinant and $\re(a)
> (m-1)\beta/2$, see \citet{gr:87}.

\section{Jacobians}\label{sec3}

We now consider several Jacobians of singular matrix transformations in terms of the $\beta$
parameter. For a detailed discussion of related issues see \citet{u:94}, \citet{dg:97},
\citet{dgm:97}, \citet{rv:05}, \citet{ihl:07}, \-\citet{dggj:11} and \citet{lx:10}.

Propositions \ref{lemsvd} and \ref{lemsd} and Corollary \ref{lemW} generalise the results in
\citet{dgm:97}, \citet{rv:05} and \citet{lx:10} obtained for the real, complex and quaternion
cases, respectively.

\begin{prop}[Singular value decomposition, $SVD$]\label{lemsvd}
Let $\mathbf{X} \in \mathfrak{L}_{m,n}^{+ \beta}(q)$, such that $\mathbf{X} =
\mathbf{V}_{1}\mathbf{DW}_{1}^{*}$ be the nonsingular part of the SVD with, $\mathbf{V}_{1} \in
{\mathcal V}_{q,n}^{\beta}$, $\mathbf{W}_{1} \in {\mathcal V}_{q,m}^{\beta}$ and $\mathbf{D} =
\diag(d_{1}, \cdots,d_{q}) \in \mathfrak{D}_{m}^{1}$, $d_{1}> \cdots > d_{q} > 0$. Then
\begin{equation}\label{svd}
    (d\mathbf{X}) = 2^{-q}\pi^{\tau} \prod_{i = 1}^{q} d_{i}^{\beta(n + m - 2q +1) -1}
    \prod_{i < j}^{q}(d_{i}^{2} - d_{j}^{2})^{\beta} (d\mathbf{D})\wedge(\mathbf{V}_{1}^{*}d\mathbf{V}_{1})\wedge
    (\mathbf{W}_{1}^{*}d\mathbf{W}_{1}),
\end{equation}
where
$$
  \tau = \left\{
             \begin{array}{rl}
               0, & \beta = 1; \\
               -\beta q/2, & \beta = 2, 4, 8.
             \end{array}
           \right.
$$
\end{prop}

\begin{prop}[Spectral decomposition, $SD$]\label{lemsd}
Let $\mathbf{S} \in \mathfrak{S}_{m}^{+\beta}(q)$. Then the nonsingular part of the spectral
decomposition can be written as $\mathbf{S} = \mathbf{W}_{1}\mathbf{\Lambda W}^{*}_{1}$, where
$\mathbf{W}_{1} \in {\mathcal V}_{q,m}^{\beta}$ and $\mathbf{\Lambda} = \diag(\lambda_{1},
\dots, \lambda_{m}) \in \mathfrak{D}_{q}^{1}$, with $\lambda_{1}> \cdots> \lambda_{q}>0$. Then
\begin{equation}\label{sd}
    (d\mathbf{S}) = 2^{-q} \pi^{\tau} \prod_{i = 1}^{q} \lambda_{i}^{\beta(m-q)}\prod_{i < j}^{q} (\lambda_{i} - \lambda_{j})^{\beta}
    (d\mathbf{\Lambda})\wedge(\mathbf{W}_{1}^{*}d\mathbf{W}_{1}),
\end{equation}
where $\tau$ is defined in Proposition \ref{lemsvd}.
\end{prop}

\begin{cor}\label{lemW}
Let $\mathbf{X} \in \mathfrak{L}_{m,n}^{+\beta}(q)$, and  $\mathbf{S} =
\mathbf{X}^{*}\mathbf{X} = \mathbf{W}_{1}\mathbf{\Lambda W}^{*}_{1} \in
\mathfrak{S}_{m}^{+\beta}(q).$ Then
\begin{equation}\label{w}
    (d\mathbf{X}) = 2^{-q} |\mathbf{\Lambda}|^{\beta(n - m + 1)/2 - 1}
    (d\mathbf{S})\wedge(\mathbf{V}_{1}^{*}d\mathbf{V}_{1}),
\end{equation}
with $\mathbf{V}_{1} \in {\mathcal V}_{n,q}^{\beta}$.
\end{cor}
\begin{proof}
The proof is obtained immediately from Propositions \ref{lemsvd} and \ref{lemsd}. \qed
\end{proof}

Let $\mathbf{A} \in \mathfrak{L}_{n,m}^{+\beta}(q)$ then $\mathbf{A}^{+} \in
\mathfrak{L}_{n,m}^{+\beta}(q)$ denotes its Moore-Penrose inverse.

\begin{thm}[Moore-Penrose inverse $\mathbf{S} = \mathbf{S}^{*}$]\label{theoi}
Let $\mathbf{S} \in \mathfrak{S}_{m}^{+\beta}(q)$ as in Proposition \ref{lemsd}. Then ignoring
the sign, if $\mathbf{V} = \mathbf{S}^{+}\in \mathfrak{S}_{m}^{+\beta}(q)$
\begin{equation}\label{gis}
    (d\mathbf{V}) = \prod_{i = 1}^{q} \lambda_{i}^{\beta(-2m +q+ 1) - 2}(d\mathbf{S}).
\end{equation}
\end{thm}
\begin{proof}
The proof follows by observing that if $\mathbf{S} = \mathbf{W}_{1} \mathbf{\Lambda}
\mathbf{W}^{*}_{1}$ is the nonsingular part of the spectral decomposition of $\mathbf{S}$, then
$\mathbf{V} = \mathbf{S}^{+} = \mathbf{W}_{1} \mathbf{\Lambda}^{-1} \mathbf{W}^{*}_{1}$.
Therefore, from (\ref{sd})
\begin{equation}\label{31}
    (d\mathbf{S}) = 2^{-q} \pi^{\tau} \prod_{i = 1}^{q} \lambda_{i}^{\beta(m-q)}\prod_{i < j}^{q}
    (\lambda_{i} - \lambda_{j})^{\beta}
    (d\mathbf{\Lambda})\wedge(\mathbf{W}_{1}^{*}d\mathbf{W}_{1}).
\end{equation}
And
$$
  (d\mathbf{V}) = 2^{-q} \pi^{\tau} \prod_{i =1}^{q} \lambda_{i}^{-\beta(m-q)-2} \prod_{i < j}
    \left(\lambda_{i}^{-1} - \lambda_{j}^{-1}\right)^{\beta}(\mathbf{W}^{*}_{1}d\mathbf{W}_{1})
    \wedge (d\mathbf{\Lambda}),
$$
taking into account that (ignoring the sign),
$$
  (d\mathbf{\Lambda}^{-1}) = \bigwedge_{i = 1}^{q} d\lambda_{i}^{-1} = \bigwedge_{i = 1}^{q} (-1)
    \frac{d \lambda_{i}}{\lambda_{i}^{2}}= \prod_{i =1}^{q} \lambda_{i}^{-2}
    \bigwedge_{i = 1}^{q} d\lambda_{i} = \prod_{i =1}^{q} \lambda_{i}^{-2}
    (d\mathbf{\Lambda}).
$$
Then,
\begin{equation}\label{32}
   (\mathbf{W}^{*}_{1}d\mathbf{W}_{1})\wedge (d\mathbf{\Lambda}) =  2^{q} \pi^{-\tau} \left [
   \prod_{i =1}^{q} \lambda_{i}^{-\beta(m-q)-2} \prod_{i < j}\left(\lambda_{i}^{-1} -
    \lambda_{j}^{-1}\right)^{\beta} \right ]^{-1} (d\mathbf{V}).
\end{equation}
Finally (ignoring the sign), observe that,
\begin{equation}\label{33}
   \prod_{i < j}\frac{\left(\lambda_{i}^{-1} - \lambda_{j}^{-1}\right)^{\beta}}{(\lambda_{i} -
   \lambda_{j})^{\beta}} = \prod_{i < j} \frac{1}{\left(\lambda_{i} \lambda_{j}\right)^{\beta}} = \prod_{i =
   1}^{q}\lambda_{i}^{-\beta(q-1)}.
\end{equation}
The desired result is obtained by substituting (\ref{32}) and (\ref{33}) in (\ref{31}). \qed
\end{proof}

\begin{thm}[Moore-Penrose inverse]\label{theoig}
Let $\mathbf{X} \in \mathfrak{L}_{n,m}^{+\beta}(q)$ as in Proposition \ref{lemsvd}. Then
ignoring the sign, if $\mathbf{Y} = \mathbf{X}^{+}$
\begin{equation}\label{gi}
    (d\mathbf{Y}) = \prod_{i = 1}^{q} d_{i}^{-2\beta(m + n -q)}(d\mathbf{X}).
\end{equation}
\end{thm}
\begin{proof}
The proof is analogous to that given for Theorem \ref{theoi}, using (\ref{svd}). \qed
\end{proof}

Theorems \ref{theoi} and \ref{theoig} were obtained by \citet{dgm:97} and \citet{lx:10} for the
real and quaternion cases, respectively.

The next result was proposed by \citet{u:94} in the real case as a conjecture. Subsequently,
\citet{dg:97} proposed an indirect proof of this conjecture. Later \citet{jdggj:09b} provided
an alternative proof based on the exterior product, also in the real case. In 2010,
\citet{lx:10} extended this result to the quaternion case, generalising the indirect proof
stated in \citet{dg:97}. We now propose two alternative statements of these results for real
normed division algebras.

\begin{thm}[First Uhlig's conjecture via $SVD$]\label{theobf}
Let $\mathbf{X}, \mathbf{Y} \in \mathfrak{S}_{m}^{+\beta}(n)$ such that $\mathbf{X} =
\mathbf{B}^{*}\mathbf{YB}$, where $\mathbf{B} \in \mathfrak{L}_{m,m}^{\beta}(m)$. In addition,
consider the nonsingular part of the SD, $\mathbf{X} = \mathbf{G}_{1} \mathbf{\Delta}
\mathbf{G}_{1}^{*}$ and $\mathbf{Y} = \mathbf{H}_{1} \mathbf{\Lambda} \mathbf{H}_{1}^{*}$,
where $\mathbf{\Delta} = \diag(\delta_{1}, \dots,\delta_{n}), \mathbf{\Lambda} =
\diag(\lambda_{1}, \dots,\lambda_{n}) \in \mathcal{D}_{n}^{1}$ and $\mathbf{G}_{1},
\mathbf{H}_{1} \in \mathcal{V}_{n,m}^{\beta}$. Then
\begin{eqnarray}\label{fucSVD1}
  (d\mathbf{X}) &=& |\mathbf{B}|^{\beta n} |\mathbf{G}^{*}_{1}\mathbf{B}^{*}\mathbf{H}_{1}|^{\beta(m-n-1) +2}
                     (d\mathbf{Y}) \\
   &=& |\mathbf{B}|^{\beta n} |\mathbf{H}^{*}_{1}\mathbf{B}\mathbf{G}_{1}|^{\beta(m-n-1) +2}
                     (d\mathbf{Y}) \\ \label{fucSVD2}
   &=& |\mathbf{B}|^{\beta n} |\mathbf{\Delta}|^{\beta(m-n-1)/2 +1}
                    |\mathbf{\Lambda}|^{-\beta(m-n-1)/2 - 1} (d\mathbf{Y}),
\end{eqnarray}
with
$$
  (d\mathbf{Y}) = 2^{-n} \pi^{\tau} \prod_{i = 1}^{n} \lambda_{i}^{\beta(m-n)}
    \prod_{i < j}^{n} (\lambda_{i} - \lambda_{j})^{\beta}
    (d\mathbf{\Lambda})\wedge(\mathbf{H}_{1}^{*}d\mathbf{H}_{1}).
$$
\end{thm}
\begin{proof}
Considering Propositions \ref{lemsvd}, \ref{lemsd} and  \citet[Proposition 1]{dggj:11} the
proof is analogous to that given in \citet{jdggj:09b} for the real case.
\end{proof}

\begin{prop}[ $QR$ decomposition, $QRD$] \label{teoqr} Let $\mathbf{X} \in
\mathfrak{L}_{n,m}^{+\beta}(q)$, then there exists a matrix $\mathbf{H}_{1} \in
\mathcal{V}_{q,n}^{\beta}$ and an upper quasi-triangular matrix  $\mathbf{T} \in
\mathfrak{T}${\tiny U}$_{m,q}^{+\beta}$ such that $\mathbf{X} = \mathbf{H}_{1}\mathbf{T}$ is
the nonsingular part of the $QR$ decomposition and
\begin{equation}
\label{qr}
   (d\mathbf{X}) = \prod_{i = 1}^{q} t_{ii}^{\beta(n-i+1)-1}(\mathbf{H}^{*}_{1}d\mathbf{H}_{1})\wedge(d\mathbf{T}).
\end{equation}
\end{prop}

\begin{prop} [Cholesky's decomposition, $CHD$]\label{teoch} Let $\mathbf{S}
\in \mathfrak{S}_{m}^{+\beta}(q)$. Then $\mathbf{S} = \mathbf{T}^{*}\mathbf{T}$, where
$\mathbf{T} \in \mathfrak{T}${\tiny U}$_{m,q}^{+\beta}$ such that
$$
  \mathbf{T} = \left (\mathbf{T}_{1} \ \mathbf{T}_{2} \right),
$$
with $\mathbf{T}_{1} \in \mathfrak{T}${\tiny U}$_{q,q}^{+\beta}$, $q \times q$ upper triangular
matrix. Also, let $\mathbf{X} \in \mathfrak{L}_{m,n}$, with $\mathbf{X} =
\mathbf{H}_{1}\mathbf{T}$ ($QR$ Decomposition) and $\mathbf{S} = \mathbf{X}^{*}\mathbf{X} =
\mathbf{T}^{*}\mathbf{T}$ (Cholesky decomposition) such that
$$
   \mathbf{S} = \left (
          \begin{array}{cc}
            \mathbf{T}^{*}_{1}\mathbf{T}_{1} &  \mathbf{T}^{*}_{1}\mathbf{T}_{2}\\
            \mathbf{T}^{*}_{2}\mathbf{T}_{1} &  \mathbf{T}^{*}_{2}\mathbf{T}_{2}\\
          \end{array}
       \right ) =
       \left (
          \begin{array}{cc}
            \build{\mathbf{S}_{11}}{}{q \times q} & \build{\mathbf{S}_{12}}{}{q \times m-q} \\
            \build{\mathbf{S}_{21}}{}{m-q \times q} & \build{\mathbf{S}_{22}}{}{m-q \times m-q} \\
          \end{array}
       \right ) \qquad \mbox{with}  \qquad \mathbf{S}_{11} \in \mathfrak{P}_{q}^{\beta}.
$$
Then
\begin{enumerate}
   \item
   \begin{equation}\label{teoch1}
    (d\mathbf{S}) = 2^{q} \displaystyle\prod_{i = 1}^{q} t_{ii}^{\beta(m - i) + 1}
    (d\mathbf{T}),
   \end{equation}
   \item Also,
    \begin{eqnarray}
      (d\mathbf{X}) &=& 2^{-q} |\mathbf{T}^{*}_{1}\mathbf{T}_{1}|^{\beta(n - m + 1)/2 - 1}
      (d\mathbf{S})\wedge(\mathbf{H}^{*}_{1}d\mathbf{H}_{1}),
      \\ \label{teoch2}
       &=& 2^{-q} |\mathbf{S}_{11}|^{\beta(n - m + 1)/2 - 1}
       (d\mathbf{S})\wedge(\mathbf{H}^{*}_{1}d\mathbf{H}_{1}. \label{teoch3}
    \end{eqnarray}
\end{enumerate}
\end{prop}

Propositions \ref{teoqr} and \ref{teoch} were studied by \citet{dggg:05b} for the real singular
case and by \citet{lx:09} for the quaternion nonsingular case.

We now study an alternative approach to Theorem \ref{theobf} with respect to an alternative
factorisation measure based on QR and the Cholesky decomposition.

\begin{thm}[First Uhlig's conjecture via $QRD$]\label{theo34} Let $\mathbf{X},
\mathbf{Y} \in \mathfrak{S}_{m}^{+\beta}(n)$, such that $\mathbf{X} =
\mathbf{B}^{*}\mathbf{YB}$, with $\mathbf{B} \in \mathfrak{L}_{m,m}^{+\beta}(m)$ fixed.
Additionally, let $\mathbf{X} = \mathbf{T}^{*}\mathbf{T}$ and $\mathbf{Y} =
\mathbf{L}^{*}\mathbf{L}$ with $\mathbf{T}, \mathbf{L} \in \mathfrak{T}${\tiny
U}$_{m,n}^{+\beta}$, such that $\mathbf{T }= (\mathbf{T}_{1} \ \mathbf{T}_{2} )$ and
$\mathbf{L} = (\mathbf{L}_{1} \ \mathbf{L}_{2})$, where $\mathbf{T}_{1}$ and $\mathbf{L}_{1}$
are $n \times n$ upper triangular matrices. Then
\begin{eqnarray}\label{teo34}
  (d\mathbf{X}) &=& |\mathbf{T}^{*}_{1}\mathbf{T}_{1}|^{\beta(m-n-1)/2 +1}|\mathbf{L}^{*}_{1}\mathbf{L}_{1}|^{-\beta(m-n-1)/2 -1}
  |\mathbf{B}|^{\beta n}(d\mathbf{Y}),
\end{eqnarray}
with
$$
   (d\mathbf{Y})= 2^{n} \prod_{i = 1}^{n} l_{ii}^{\beta(m - i) + 1} (d\mathbf{L}).
$$
\end{thm}
\begin{proof} Let $\mathbf{Z} \in \mathfrak{L}_{n,m}^{+\beta}(n)$, such that $\mathbf{Y} =
\mathbf{Z}^{*}\mathbf{Z}$. Then
\begin{equation}\label{eq1teo34}
    \mathbf{X} = \mathbf{B}^{*}\mathbf{YB} = \mathbf{B}^{*}\mathbf{Z}^{*}\mathbf{ZB} =
    \mathbf{\Phi}^{*}\mathbf{\Phi}, \quad \mbox{with}\quad \mathbf{\Phi} =\mathbf{ZB}.
\end{equation}
from Proposition \ref{teoch}, equation \ref{teoch2}
\begin{equation}
    (d\mathbf{\Phi}) =  2^{-n} |\mathbf{T}^{*}_{1}\mathbf{T}_{1}|^{\beta(n - m + 1)/2 - 1} (\mathbf{Q}^{*}d\mathbf{Q})(d\mathbf{X}).
\end{equation}
In which $\mathbf{\Phi} = \mathbf{QT}$ with $\mathbf{T} \in \mathfrak{T}${\tiny U}$_{m,
n}^{+\beta}$, $\mathbf{Q} \in \mathfrak{U}^{\beta}(n)$, and $\mathbf{X} =
\mathbf{\Phi}^{*}\mathbf{\Phi} = \mathbf{T}^{*}\mathbf{T}$. Then
\begin{equation}\label{eq2teo34}
    (d\mathbf{X}) =  2^{n} |\mathbf{T}^{*}_{1}\mathbf{T}_{1}|^{-\beta(n - m + 1)/2 - 1} (\mathbf{Q}^{*}d\mathbf{Q})^{-1}(d\mathbf{\Phi}).
\end{equation}
Note that $d\mathbf{\Phi} = d\mathbf{ZB}$, and then by \citet[Proposition 1]{dggj:11} we have
that $(d\mathbf{\Phi}) = |\mathbf{B}|^{\beta n}(d\mathbf{Z})$, from which, substituting in
(\ref{eq2teo34}), we obtain
\begin{equation}\label{eq3teo34}
    (d\mathbf{X}) =  2^{n} |\mathbf{T}^{*}_{1}\mathbf{T}_{1}|^{-\beta(n - m + 1)/2 - 1} (\mathbf{Q}^{*}
    d\mathbf{Q})^{-1}|\mathbf{B}|^{\beta n}(d\mathbf{Z}).
\end{equation}
Now $\mathbf{Y} = \mathbf{Z}^{*}\mathbf{Z} = \mathbf{L}^{*}\mathbf{L}$ with  $\mathbf{L} \in
\mathfrak{T}_{m, n}^{+ \beta}$; from Lemma \ref{teoch}
\begin{equation}\label{eq4teo34}
    (d\mathbf{Z}) = 2^{-n} |\mathbf{L}^{*}_{1}\mathbf{L}_{1}|^{\beta(n - m + 1)/2 - 1} (d\mathbf{Y})(\mathbf{H}^{*}d\mathbf{H})
\end{equation}
where $\mathbf{Z} =\mathbf{HL}$, with $\mathbf{L} \in \mathfrak{T}${\tiny U}$^{^+\beta}_{m, n}$
and $\mathbf{H} \in \mathfrak{U}^{\beta}(n)$. Moreover, note that, due to the uniqueness of
Haar's measure, $(\mathbf{Q}^{*}d\mathbf{Q}) = (\mathbf{H}^{*}d\mathbf{H})$, see \citet{j:54}.
Thus, substituting (\ref{eq4teo34}) in (\ref{eq3teo34}), we obtain
\begin{eqnarray*}
  (d\mathbf{X}) &=& |\mathbf{L}^{*}_{1}\mathbf{L}_{1}|^{\beta(n - m + 1)/2 - 1}
        |\mathbf{T}^{*}_{1}\mathbf{T}_{1}|^{-\beta(n - m + 1)/2 - 1} |\mathbf{B}|^{\beta n} (d\mathbf{Y}) \\
      &=& |\mathbf{T}^{*}_{1}\mathbf{T}_{1}|^{\beta(m - n + 1)/2} |\mathbf{L}^{*}_{1}\mathbf{L}_{1}|^{\beta-(m- n + 1)/2}
      |\mathbf{B}|^{\beta n} (d\mathbf{Y}).
\end{eqnarray*} \qed
\end{proof}

The following result combines Theorems \ref{theoi} and \ref{theobf}, which enables us to study
the multivariate F and beta distributions (also termed matrix multivariate beta type II and
type I, respectively) for real normed division algebras in a class of singularity.

\begin{thm}\label{theosbf}
Let $\mathbf{X}, \mathbf{Y} \in \mathfrak{S}_{m}^{+\beta}(n)$ such that $\mathbf{X} =
\mathbf{B}^{*}\mathbf{Y}^{+}\mathbf{B}$, where $\mathbf{B} \in \mathfrak{L}_{m,m}^{\beta}(m)$.
In addition, consider the nonsingular part of the SD, $\mathbf{X} = \mathbf{G}_{1}
\mathbf{\Delta} \mathbf{G}_{1}^{*}$ and $\mathbf{Y} = \mathbf{H}_{1} \mathbf{\Lambda}
\mathbf{H}_{1}^{*}$, where $\mathbf{G}_{1}, \mathbf{H}_{1} \in \mathcal{V}_{n,m}^{\beta}$ and
$\mathbf{\Delta} = \diag(\delta_{1}, \dots,\delta_{n}), \mathbf{\Lambda} = \diag(\lambda_{1},
\dots,\lambda_{n}) \in \mathcal{D}_{n}^{1}$. Then
\begin{equation}\label{fuciSVD}
    (d\mathbf{X}) = |\mathbf{B}|^{\beta n} |\mathbf{\Delta}|^{\beta(m-n-1)/2 +1}
                    |\mathbf{\Lambda}|^{-\beta(3m-n-1)/2 - 1} (d\mathbf{Y}),
\end{equation}
with
$$
  (d\mathbf{Y}) = 2^{-n} \pi^{\tau} \prod_{i = 1}^{n} \lambda_{i}^{\beta(m-n)}
    \prod_{i < j}^{n} (\lambda_{i} - \lambda_{j})^{\beta}
    (d\mathbf{\Lambda})\wedge(\mathbf{H}_{1}^{*}d\mathbf{H}_{1}).
$$
\end{thm}
\begin{proof} Let $\mathbf{Z} = \mathbf{Y}^{+}$ and observe that $\mathbf{Z} = \mathbf{H}_{1} \mathbf{\Lambda}^{-1}
\mathbf{H}_{1}^{*}$. Then by Theorem \ref{theobf},
\begin{equation}\label{neq1fb}
    (d\mathbf{X}) = |\mathbf{B}|^{\beta n} |\mathbf{\Delta}|^{\beta(m-n-1)/2 +1}
                    \left|\mathbf{\Lambda}^{-1}\right|^{-\beta(m-n-1)/2 - 1} (d\mathbf{Z}).
\end{equation}
Now by Theorem \ref{theoi}
\begin{equation}\label{neq2fb}
    (d\mathbf{Z}) = \prod_{i = 1}^{n} \lambda_{i}^{\beta(-2m +n+ 1) - 2}(d\mathbf{Y}).
\end{equation}
The result follows by substituting (\ref{neq2fb}) into (\ref{neq1fb}). \qed
\end{proof}
Observe that this result corrects an erratum in the exponent of the determinant of
$\mathbf{\Lambda}$ in \citet[Theorema 2.2]{jdggj:09b}, obtained in the real case.

\section*{Conclusions}
Note that the results presented here contain as special cases all versions of the results
previously obtained in real, complex and quaternion and octonion cases,  both for singular and
nonsingular cases. For example from Theorem \ref{theo34},
\begin{equation}\label{eq1c}
    (d\mathbf{X}) = \left(|\mathbf{T}^{*}_{1}\mathbf{T}_{1}||\mathbf{L}^{*}_{1}\mathbf{L}_{1}|^{-1}\right)^{\beta(m-n-1)/2 +1}
  |\mathbf{B}|^{\beta n}(d\mathbf{Y}),
\end{equation}
Now, let $\mathbf{X}$ and $\mathbf{Y}$ be nonsingular matrices, that is, $m = n$, therefore
$\mathbf{T}_{1}$ and $\mathbf{L}_{1}$ are square nonsingular triangular matrices and
$$
  |\mathbf{T}^{*}_{1}\mathbf{T}_{1}| = |\mathbf{X}| = |\mathbf{B}^{*}\mathbf{YB}| =
  |\mathbf{B}^{*}\mathbf{L}^{*}_{1}\mathbf{L}_{1}\mathbf{B}| = |\mathbf{B}|^{2}|\mathbf{L}^{*}_{1}\mathbf{L}_{1}|
$$
then,
\begin{equation}\label{eq2c}
   |\mathbf{B}|^{2} = |\mathbf{T}^{*}_{1}\mathbf{T}_{1}|
   |\mathbf{L}^{*}_{1}\mathbf{L}_{1}|^{-1}.
\end{equation}
The desired result
$$
  (d\mathbf{X}) = |\mathbf{B}|^{\beta(m-1)+2}(d\mathbf{Y}),
$$
is obtained by substituting (\ref{eq2c}) into (\ref{eq1c}), see \citet{m:97} for real and
complex cases and \citet{lx:10} for the quaternion case, among others.

Finally, observe that the real dimension of real normed division algebras can be expressed as
potentia of 2, $\beta = 2^{n}$ for $n = 0,1,2,3$. Moreover, as observed by \citet{k:84}, the
results obtained in this work can be extended to hypercomplex cases; that is, for complex,
bicomplex, biquaternion and bioctonion (or sedenionic) algebras, which of course are not
division algebras (except the complex algebra), but are Jordan algebras, together with all
their isomorphic algebras. Note, too, that hypercomplex algebras are obtained by replacing the
real numbers with complex numbers in the construction of real normed division algebras. Thus,
the results for hypercomplex algebras are obtained by simply replacing $\beta$ with $2\beta$ in
our results (we reiterate, as reported by \citet{k:84}).

\section*{Acknowledgements}

This paper was written during J. A. D\'{\i}az-Garc\'{\i}a's stay as a visiting professor at the
Department of Statistics and O. R. of the University of Granada, Spain.

\end{document}